\documentclass[12pt]{amsart}

\usepackage{amssymb, latexsym, amsmath, extarrows, mathrsfs, amsthm,color, amsrefs} 
\usepackage{graphicx}
\usepackage{courier}
\usepackage{xcolor}
\usepackage{MnSymbol}
\usepackage{enumitem}

\usepackage[margin=1in, centering]{geometry}
\usepackage[bookmarksnumbered, colorlinks, plainpages]{hyperref}
\usepackage{hyperref} 
\hypersetup{
    colorlinks=true,       
    linkcolor=blue,          
    citecolor=magenta,        
    filecolor=magenta,      
    urlcolor=cyan           
}

\usepackage{mathtools}
\mathtoolsset{showonlyrefs,showmanualtags}

\newtheorem{theorem}{Theorem}[section]

\newtheorem{lemma}[theorem]{Lemma}
\newtheorem{proposition}[theorem]{Proposition}
\newtheorem{remark}[theorem]{Remark}

\newtheorem{definition}{Definition}[section]

\newcommand{\e}{\varepsilon}

\newcommand{\sj}{\mathscr{J}}

\newcommand{\Z}{\mathbb Z}
\newcommand{\R}{\mathbb R}

\newcommand{\T}{\mathbb T}
\newcommand{\N}{\mathbb N}

\newcommand{\eqdef}{\overset{\mathrm{def}}{=\joinrel=}}

\definecolor{deepgreen}{cmyk}{1,0,1,0.5}

\def\les{\;\mathrm{\lesssim}\;}

\title[A.c. spectrum of limit-periodic Jacobi matrices]{A note on the absolutely continuous spectrum of limit-periodic Jacobi matrices}

\author[M.\ Solis]{Moises Gomez Solis}
\address{Department of Mathematics \\ Louisiana State University  \\  Baton Rouge, LA 70803, USA}
\email{mgome29@lsu.edu}

\thanks{M.G. Solis was partially supported by NSF-DMS-2143369.}

\begin{document}

\begin{abstract}
We show that limit-periodic Jacobi matrices in the weak exponential class $EC(\eta)$ have purely absolutely continuous spectral type when $\eta$ is sufficiently large, depending only on the $R$-boundedness constant. The proof first establishes homogeneity of the spectrum from the weak exponential approximation rate. A closure theorem for reflectionless Jacobi matrices then transfers reflectionlessness from the periodic approximants to the limit, and a theorem of Poltoratski and Remling rules out singular spectrum on a homogeneous set. This paper may be viewed as a sequel to our earlier manuscript on limit-periodic Jacobi matrices; see \cite{HaSo}.
\end{abstract}

\date{\today}

\maketitle

\section{Introduction}


We study the spectral type of limit-periodic Jacobi matrices, motivated by their connection to quantum dynamics. Here, by a \textit{Jacobi matrix} we mean a bounded self-adjoint operator $\sj = \sj(\{a_n\},\{b_n\})$, defined by
    \begin{equation}
        (\sj \phi)_n = a_{n-1}\phi_{n-1} + b_n\phi_n + a_{n}\phi_{n+1}, \quad\quad n\in \Z, \quad\quad  \text{ acting on }\phi\in\ell^2(\Z),  
    \end{equation}
where $a \equiv \{a_n\}_{n=-\infty}^\infty, b\equiv \{b_n\}_{n=-\infty}^\infty \in \ell^\infty(\Z,\R)$ and $\inf_n a_n >0$.

In the case of one-dimensional Jacobi matrices, it is well known that periodic Jacobi coefficients give purely absolutely continuous spectrum and strong ballistic transport; this holds even when the Jacobi coefficients are matrices with complex entries; see \cite{DMY}. We say that $\sj$ is $q$-periodic if its coefficients $\{a_n\}_{n\in\Z}$ and $\{b_n\}_{n\in\Z}$ are $q$-periodic; that is,
\begin{equation}
    a_{n+q} = a_n, \qquad b_{n+q}=b_n, \qquad \forall \,\,n\in\Z.
\end{equation}

A natural question, studied in \cite{Fil2017}, is under what conditions
ballistic transport and purely absolutely continuous spectrum extend to
limit-periodic Jacobi matrices. Here we focus on the spectral question.
As in \cite{HaSo, Fil2017}, we impose an \(R\)-boundedness assumption on
\(\sj\): there is an \(R\geq1\) such that, for every \(n\in\Z\), the
Jacobi coefficients satisfy
\begin{equation*}
    \frac{1}{R}\leq a_n \leq R, \qquad \text{ and } \qquad |b_n|\leq R.
\end{equation*}
\begin{definition}
    If a Jacobi matrix satisfies the $R$-boundedness assumption, we write $\sj\in\mathcal J(R)$ to indicate membership in the set of $R$-bounded Jacobi matrices.
\end{definition}
\begin{remark}
    The study of ballistic transport is not exclusive to the almost-periodic setting. In recent years, there has been increasing interest in understanding this type of dynamics. Examples include the periodic case \cite{DYL}, the almost-periodic continuum case \cite{Young}, and the quasi-periodic case (for Schr\"odinger operators) \cite{Ge2000}. For a survey of the questions surrounding ballistic transport and the relevant literature, see \cite{DMY}.
\end{remark}

\begin{definition}
    We say that a Jacobi matrix $\sj$ is \textit{limit-periodic} if there exists a sequence $\{\sj_{q_n}\}_{n=0}^{\infty}$ of $q_n$-periodic matrices such that
    \begin{equation}
        \lim_{n\rightarrow\infty} \|\sj - \sj_{q_n}\|_{op} = 0;
    \end{equation}
    here $\|\cdot\|_{op}$ denotes the operator norm.
\end{definition}
The following definition specifies the weak exponential class used in this paper.

\begin{definition}[Exponential class]
    We say that $\sj\in \text{EC}(\eta)$, for a given $\eta>0$, if there are integers $\{q_n\}_{n=0}^\infty \subseteq \N_{\geq1}$, with $q_n\mid q_{n+1}$ and $q_n\neq q_{n+1}$, such that $\sj_{q_n}$ is $q_n$-periodic and 
    \begin{equation}
        \lim_{n\rightarrow\infty} q_{n+1}^3e^{\eta q_n}\|\sj-\sj_{q_n}\|= 0.
    \end{equation}
\end{definition}

The main result is the following theorem.

\begin{theorem}\label{thm:main}
For every $R\geq1$, there exists $\eta_0=\eta_0(R)>0$ such that if $\sj\in EC(\eta)\cap\mathcal J(R)$ with $\eta>\eta_0$, then $\sj$ has purely absolutely continuous spectral type.
\end{theorem}

\subsection{Structure of the paper}

The proof of Theorem \ref{thm:main} is organized as follows. In Section 2 we prove homogeneity of $\Sigma=\sigma(\sj)$ directly from the weak exponential class condition. The key point is to use the folded band scale inherited from the preceding periodic approximant, rather than the much smaller crude band scale associated with the full period. In Section 3 we use stability of the reflectionless property under periodic approximation and then combine reflectionlessness with homogeneity to rule out singular spectrum.

\section{Homogeneity of the spectrum}
\begingroup

Throughout this section, $|\cdot|$ denotes one-dimensional Lebesgue measure. We also use the notation
\[
B_\delta(E)=(E-\delta,E+\delta).
\]

We start by recalling the geometric notion of the Hausdorff metric.

\begin{definition}[Hausdorff metric]
    Let $B$ and $G$ be two nonempty compact subsets of $\R$. We denote their Hausdorff distance by
    \begin{equation}
        d_H(B,G)\eqdef \inf \{\e >0 \colon B\subseteq \mathscr{B}_{\e}(G) \text{ and }G\subseteq\mathscr{B}_{\e}(B)\},
    \end{equation}
    where
    \begin{equation}
        \mathscr{B}_\e(X) \eqdef \bigcup_{x\in X}\{y \in \R \colon |x-y|< \e\}.
    \end{equation}
\end{definition}

\begin{definition}[Homogeneity]\label{def: homogeneity}
    Let $F \Subset \R$ be compact. We say that $F$ is homogeneous if there are constants $\delta_0>0$ and $\tau>0$ such that
    \begin{equation}
        |F\cap B_\delta(\lambda)|\geq \tau\delta
    \end{equation}
    for every $\lambda\in F$ and every $0<\delta\leq\delta_0$.
\end{definition}

We will use two standard facts.

\begin{proposition}\label{prop: prop 2.1 in Filman} \cite[Proposition 2.1]{FilHom}
Let $\{B_n\}_{n=1}^{\infty}$ and $\{G_n\}_{n=1}^\infty$ be sequences of compact subsets of $\R$ that converge in the Hausdorff metric. Write $B = \lim_{n\rightarrow\infty} B_n$ and $G = \lim_{n\rightarrow\infty} G_n$. Then
\begin{equation}
    |B\cap G| \geq \limsup_{n\rightarrow\infty} |B_n\cap G_n|.
\end{equation}
\end{proposition}

\begin{proposition}\label{prop: prop 2.2 in Filman}\cite[Proposition 2.2]{FilHom}
    Let $A$ and $B$ be bounded self-adjoint operators on a Hilbert space $\mathcal{H}$. Then
    \begin{equation}
        d_H(\sigma(A),\sigma(B)) \leq \|A-B\|_{\text{op}}.
    \end{equation}
\end{proposition}

\begin{remark}
    The Pastur--Tkachenko condition $\|\sj-\sj_{q_n}\|\lesssim e^{-\eta q_{n+1}}$ immediately gives the summability used in \cite{FilHom}. The weaker class used here only gives decay at the previous scale $q_{n-1}$. The proof below therefore uses the folded band scale
    \begin{equation}
        e^{-c_1q_{n-1}}\left(\frac{q_{n-1}}{q_n}\right)^2
    \end{equation}
    instead of the crude scale $e^{-c_1q_n}$.
\end{remark}

\begin{lemma}[Floquet band structure]\label{lem:floquet-band-structure}
Let $J$ be a scalar $p$-periodic Jacobi matrix with positive
off-diagonal coefficients, let $J(\theta)$ denote its $p\times p$
Floquet fiber with multiplier $e^{2\pi i\theta}$, and let
$\Delta_p$ be the trace of the determinant-one monodromy matrix.
Then
\begin{equation}\label{eq:floquet-spectrum-discriminant}
    \sigma(J)=\{E\in\R:|\Delta_p(E)|\leq2\}.
\end{equation}
If
\[
    e_1\leq e_2\leq\cdots\leq e_{2p}
\]
is the combined list of the eigenvalues of $J(0)$ and
$J(1/2)$, in increasing order and with multiplicity, then
\begin{equation}\label{eq:ordered-band-edges}
    \sigma(J)=\bigcup_{j=1}^p[e_{2j-1},e_{2j}].
\end{equation}
Thus closed gaps correspond to equal consecutive edge values.  On the
interior of every band, $\Delta_p'$ does not vanish, and the band has a
monotone $C^1$ parametrization $\lambda_j:(0,1/2)\to\R$ satisfying
\begin{equation}\label{eq:floquet-parametrization}
    \Delta_p(\lambda_j(\theta))=2\cos(2\pi\theta).
\end{equation}
The parametrization extends continuously to $[0,1/2]$, and its
interior derivative has continuous one-sided extensions at both endpoints.
No two-sided differentiability of the globally ordered eigenvalue list is
asserted at a closed gap.
\end{lemma}

\begin{proof}
These are standard facts of scalar periodic Floquet theory; see also
\cite[Theorem 2.3]{FilHom}.  For completeness, we give the details
needed below.  The Floquet boundary condition with multiplier
$z=e^{2\pi i\theta}$ is equivalent to saying that $z$ is an
eigenvalue of the monodromy matrix.  Since the latter has determinant
one, this is equivalent to
\[
    \Delta_p(E)=z+z^{-1}=2\cos(2\pi\theta).
\]
Equivalently, the standard cyclic determinant identity gives
\begin{equation}\label{eq:cyclic-determinant-identity}
    \det(E-J(\theta))
    =c_p\bigl(\Delta_p(E)-2\cos(2\pi\theta)\bigr),
\end{equation}
where $c_p\neq0$ is independent of $E$ and $\theta$.  (The identity
follows directly by expanding the cyclic tridiagonal determinant; its
constant is immaterial here.)  Taking the union over $\theta$ proves
\eqref{eq:floquet-spectrum-discriminant}.

At $\theta=0$ and $\theta=1/2$, identity
\eqref{eq:cyclic-determinant-identity} shows that the periodic and
antiperiodic eigenvalues, counted with multiplicity, are exactly the
roots of $\Delta_p(E)-2$ and $\Delta_p(E)+2$.  The polynomial
$\Delta_p$ has degree $p$ and $|\Delta_p(E)|>2$ for $|E|$ large.
Moreover, up to a positive constant,
\[
    \Delta_p(E)^2-4=\prod_{j=1}^{2p}(E-e_j),
\]
with roots listed according to multiplicity.  Since this polynomial is
positive near both ends of the real line, its nonpositive set is
exactly the union of the alternating intervals in
\eqref{eq:ordered-band-edges}.  A double root produces a closed gap and
is retained twice in the ordered list.

Fix $0<\theta<1/2$ and put $z=e^{2\pi i\theta}$.  Since $z\neq z^{-1}$,
the $z$-eigenspace of the determinant-one monodromy matrix is
one-dimensional.  Cauchy data identify the eigenspace of the Floquet
fiber $J(\theta)$ at energy $E$ with this monodromy eigenspace.
Therefore every eigenvalue of the Hermitian matrix $J(\theta)$ is
simple.  By \eqref{eq:cyclic-determinant-identity}, every root of
$\Delta_p(E)-2\cos(2\pi\theta)$ is consequently simple.  Hence
$\Delta_p'(E)\neq0$ whenever $|\Delta_p(E)|<2$, that is, in the
interior of each individual band away from its endpoints.  The
derivative has constant sign on each such band interior, so
$\Delta_p$ can be inverted
there.  This gives the monotone parametrization
\eqref{eq:floquet-parametrization}; continuity supplies the endpoint
extension.  The Floquet fibers form a finite-dimensional real-analytic
Hermitian family.  By the standard finite-dimensional Rellich theorem \cite{Rellich},
near an endpoint their eigenvalues admit real-analytic labels.  Simplicity
on the open half-zone makes each band branch agree there with one of these
labels.  Consequently, its derivative extends continuously from that side,
including when a gap is closed.
\end{proof}

\begin{lemma}[Uniform periodic band velocity]\label{lem: uniform periodic band velocity}
For every $R\geq1$, there are constants $c_R>0$ and
$c_1=c_1(R)>0$ with the following property.  If
$J\in\mathcal J(R)$ is $p$-periodic and $\lambda_j(\theta)$ is one
of the Floquet band parametrizations from Lemma
\ref{lem:floquet-band-structure}, then, for $0<\theta<1/2$,
\begin{equation}\label{eq: uniform periodic band velocity}
    |\dot\lambda_j(\theta)|
    \geq c_R e^{-c_1p}|\sin(2\pi\theta)|.
\end{equation}
The same estimate holds at an endpoint in the one-sided limiting
sense.
\end{lemma}

\begin{proof}
For the Jacobi equation $Ju=Eu$, use the transfer relation
\[
    \binom{u_{n+1}}{a_nu_n}
    =A_n(E)\binom{u_n}{a_{n-1}u_{n-1}},
    \qquad
    A_n(E)=\frac1{a_n}
    \begin{pmatrix}
        E-b_n&-1\\
        a_n^2&0
    \end{pmatrix}.
\]
Thus $\det A_n(E)=1$ and
\[
    M_p(E)=A_p(E)\cdots A_1(E),
    \qquad
    \Delta_p(E)=\operatorname{tr}M_p(E).
\]
Since $\|J\|\leq3R$, every $E\in\sigma(J)$ satisfies
$|E|\leq3R$.  The bounds $R^{-1}\leq a_n\leq R$ and
$|b_n|\leq R$ therefore give a constant $K_R\geq1$, independent of
$p$, such that
\[
    \|A_n(E)\|+\|\partial_EA_n(E)\|\leq K_R
    \qquad(E\in\sigma(J)).
\]
By the product rule and $|\operatorname{tr}X|\leq2\|X\|$,
\[
    |\Delta_p'(E)|\leq 2pK_R^p\leq e^{c_1p}
    \qquad(E\in\sigma(J))
\]
after increasing $c_1=c_1(R)$.  Lemma
\ref{lem:floquet-band-structure} justifies differentiating
\eqref{eq:floquet-parametrization} on a band interior.  Hence
\[
    |\dot\lambda_j(\theta)|
    =\frac{4\pi|\sin(2\pi\theta)|}
           {|\Delta_p'(\lambda_j(\theta))|}
    \geq 4\pi e^{-c_1p}|\sin(2\pi\theta)|.
\]
Reducing the constant gives the stated form.  The endpoint statement
follows by taking one-sided limits.
\end{proof}

\begin{lemma}[Stability of ordered periodic band edges]
\label{lem:ordered-edge-stability}
Let $K$ and $L$ be scalar $q$-periodic Jacobi matrices, and suppose
\[
    \|K-L\|\leq\gamma.
\]
Let
\[
    e_1(K)\leq\cdots\leq e_{2q}(K),
    \qquad
    e_1(L)\leq\cdots\leq e_{2q}(L)
\]
be their ordered periodic--antiperiodic edge lists, with multiplicity.
Then
\begin{equation}\label{eq:edge-list-stability}
    |e_j(K)-e_j(L)|\leq\gamma,
    \qquad j=1,\dots,2q.
\end{equation}
Consequently, the length of the $j$th band of $L$ is at least the
length of the $j$th band of $K$ minus $2\gamma$.
\end{lemma}

\begin{proof}
Use the same $q$-site cell for the Floquet decompositions of $K$ and
$L$.  Since $K-L$ is $q$-periodic, the direct-integral decomposition
and continuity of the fibers give
\[
    \max_{\theta\in\T}\|K(\theta)-L(\theta)\|
    =\|K-L\|\leq\gamma.
\]
The minimax principle therefore shows that each ordered eigenvalue of the
periodic fiber changes by at most $\gamma$, and the same is true for the
antiperiodic fiber.  Thus the two combined edge multisets admit a
bijection under which paired values differ by at most $\gamma$.

We record the elementary sorting argument.  If
$x_1\leq\cdots\leq x_N$ and $y_1\leq\cdots\leq y_N$ admit such a
bijection, then $|x_j-y_j|\leq\gamma$ for every $j$.  Indeed, if
$y_j<x_j-\gamma$, then the first $j$ values of the $y$-list must be
paired with $j$ values of the $x$-list strictly smaller than $x_j$,
which is impossible.  Interchanging $x$ and $y$ gives the opposite
inequality.  This proves \eqref{eq:edge-list-stability}, including at
double edges.  The last assertion follows from
Lemma \ref{lem:floquet-band-structure}, since the $j$th band has
endpoints $e_{2j-1}$ and $e_{2j}$.
\end{proof}

\begin{lemma}[Folded subband lower bound]\label{lem: folded subband lower bound}
Let $J^{(p)}\in\mathcal J(R)$ be $p$-periodic, and let $q=\ell p$
with $\ell\in\N$. Regard $J^{(p)}$ artificially as a $q$-periodic
Jacobi matrix, denoted by $\widetilde J^{(q)}$. Then every
$q$-periodic spectral band of $\widetilde J^{(q)}$ has length at
least
\begin{equation}\label{eq: artificial folded lower bound}
    c_R e^{-c_1p}\ell^{-2},
\end{equation}
where $c_1=c_1(R)$ and $c_R>0$ depend only on $R$.  Bands are counted
with multiplicity, so closed gaps are allowed.

Consequently, if $J^{(q)}\in\mathcal J(R)$ is $q$-periodic and
\begin{equation}
    \|J^{(q)}-\widetilde J^{(q)}\|\leq \gamma
\end{equation}
with
\begin{equation}
    \gamma\leq \frac14 c_R e^{-c_1p}\ell^{-2},
\end{equation}
then every $q$-periodic spectral band of $J^{(q)}$ has length at least
\begin{equation}\label{eq: perturbed folded lower bound}
    \frac12 c_R e^{-c_1p}\ell^{-2}.
\end{equation}
\end{lemma}

\begin{proof}
Let $\Delta_p$ be the $p$-periodic discriminant of $J^{(p)}$, and
fix one of its band parametrizations $E=\lambda_j(\theta)$ from
Lemma \ref{lem:floquet-band-structure}.  By Lemma
\ref{lem: uniform periodic band velocity},
\begin{equation}\label{eq: old scale band velocity}
    |\dot\lambda_j(\theta)|
    \geq c_Re^{-c_1p}|\sin(2\pi\theta)|.
\end{equation}
When $J^{(p)}$ is regarded as periodic with period $q=\ell p$, its $q$-step
monodromy is the $\ell$th power of its $p$-step monodromy.  Hence
\begin{equation}\label{eq: Chebyshev folding}
    \Delta_q(E)=2T_\ell\!\left(\frac{\Delta_p(E)}2\right),
\end{equation}
where $T_\ell$ is the Chebyshev polynomial.  The artificial
$q$-periodic edges inside the fixed $p$-band therefore occur at
\[
    \frac{\Delta_p(E)}2=\cos\left(\frac{\pi r}{\ell}\right),
    \qquad r=0,1,\dots,\ell,
\]
or, equivalently, at $\theta=r/(2\ell)$.  Thus the $p$-band is
subdivided, with multiplicity, into the $\ell$ artificial $q$-bands
\[
    \lambda_j\!\left(
       \left[\frac{r}{2\ell},\frac{r+1}{2\ell}\right]
    \right),
    \qquad r=0,\dots,\ell-1.
\]
Their lengths are bounded below by
\begin{align}
    c_Re^{-c_1p}
    \int_{r/(2\ell)}^{(r+1)/(2\ell)}
       |\sin(2\pi\theta)|\,d\theta.
\end{align}
The minimum of these integrals occurs at an endpoint interval and,
for $\ell\geq2$, equals
\[
    \frac{1-\cos(\pi/\ell)}{2\pi}\geq c\ell^{-2}.
\]
The case $\ell=1$ is immediate.  Absorbing the absolute constant into
$c_R$ proves \eqref{eq: artificial folded lower bound}.

Apply Lemma \ref{lem:ordered-edge-stability} to
$K=\widetilde J^{(q)}$ and $L=J^{(q)}$.  Every band length can decrease
by at most $2\gamma$, so every band of $J^{(q)}$ has length at least
\[
    c_R e^{-c_1p}\ell^{-2}-2\gamma.
\]
The assumed bound on $\gamma$ gives
\eqref{eq: perturbed folded lower bound}.
\end{proof}

\begin{lemma}[Local spectral loss]\label{lem: local spectral loss}
Let $J_n$ and $J_{n-1}$ be periodic Jacobi matrices with spectra $\Sigma_n$ and $\Sigma_{n-1}$. Set
\begin{equation}
    \gamma_n:=\|J_n-J_{n-1}\|
\end{equation}
and suppose that every spectral band of $\Sigma_n$ has length at least $m_n>0$. Then, for every bounded interval $I\subset\R$,
\begin{equation}\label{eq: local spectral loss}
    |I\cap(\Sigma_{n-1}\setminus\Sigma_n)|
    \leq C\left(\frac{|I|}{m_n}+1\right)\gamma_n,
\end{equation}
where $C$ is an absolute constant.
\end{lemma}

\begin{proof}
By Proposition \ref{prop: prop 2.2 in Filman},
\begin{equation}
    d_H(\Sigma_{n-1},\Sigma_n)\leq\gamma_n.
\end{equation}
Hence every point of $\Sigma_{n-1}\setminus\Sigma_n$ lies within
distance $\gamma_n$ of an endpoint of a band of $\Sigma_n$. If
\(\gamma_n=0\), the spectra are equal and there is nothing to prove.
Assume henceforth that \(\gamma_n>0\), and put
\[
    I_{2\gamma_n}=\{x\in\R:\operatorname{dist}(x,I)<2\gamma_n\}.
\]
If \(|I|\leq2\gamma_n\), then the left-hand side of
\eqref{eq: local spectral loss} is at most \(|I|\leq2\gamma_n\), and
the desired estimate is immediate.  Suppose instead that
\(|I|>2\gamma_n\).  Then \(|I_{2\gamma_n}|<3|I|\).  Every endpoint
whose \(\gamma_n\)-neighborhood meets \(I\) belongs to
\(I_{2\gamma_n}\).  Since the band interiors are disjoint and every
band has length at least \(m_n\), all but at most two of the bands
having an endpoint in \(I_{2\gamma_n}\) are contained in
\(I_{2\gamma_n}\). Thus the
number of relevant bands, and hence the number of relevant endpoints,
is at most
\[
    C\left(\frac{|I|}{m_n}+1\right).
\]
The set \(I\cap(\Sigma_{n-1}\setminus\Sigma_n)\) is contained in the
corresponding endpoint neighborhoods, each of length
\(O(\gamma_n)\).  This proves \eqref{eq: local spectral loss}.
\end{proof}

\begin{theorem}[Homogeneity of the exponential class]\label{Th: Homogeneity}
For every $R\geq1$, there exists $\eta_0=\eta_0(R)>0$ such that if $\sj\in EC(\eta)\cap\mathcal J(R)$ with $\eta>\eta_0$, then $\Sigma=\sigma(\sj)$ is homogeneous in the sense of Definition \ref{def: homogeneity}.
\end{theorem}

\begin{proof}
Let
\[
    \sj_n\eqdef\sj_{q_n},\qquad
    \Sigma_n\eqdef\sigma(\sj_n),\qquad
    \varepsilon_n\eqdef\|\sj-\sj_n\|.
\]
Operator-norm convergence controls the Jacobi coefficients entrywise:
\[
    |b_k^{(n)}-b_k|
    =|\langle\delta_k,(\sj_n-\sj)\delta_k\rangle|
    \leq\|\sj_n-\sj\|,
\]
and similarly
$|a_k^{(n)}-a_k|\leq\|\sj_n-\sj\|$.  Since
\(\sj\in\mathcal J(R)\), after discarding finitely many terms so that
\(\varepsilon_n\leq(2R)^{-1}\), we have
\(\sj_n\in\mathcal J(2R)\) for every \(n\).  In what
follows, the constants \(c_1\) and \(c_R\) are those of Lemma
\ref{lem: folded subband lower bound} for the bound \(2R\). Put
\begin{equation}
    \ell_n\eqdef\frac{q_n}{q_{n-1}},
    \qquad
    \gamma_n\eqdef\|\sj_n-\sj_{n-1}\|.
\end{equation}
The proper divisibility $q_{n-1}\mid q_n$ implies
$q_n\geq2q_{n-1}$.  Since $\sj\in EC(\eta)$, after passing to a further tail and relabeling, we have
\begin{equation}\label{eq: weak EC consecutive gamma}
    \gamma_n\leq \frac{e^{-\eta q_{n-1}}}{q_n^3}.
\end{equation}
Indeed, on a sufficiently late tail,
\[
    q_{n+1}^3e^{\eta q_n}\varepsilon_n\leq\frac12
    \quad\text{for every }n.
\]
Consequently,
\[
\begin{split}
    \gamma_n
    &\leq\varepsilon_n+\varepsilon_{n-1}\\
    &\leq
    \frac12\frac{e^{-\eta q_n}}{q_{n+1}^3}
    +\frac12\frac{e^{-\eta q_{n-1}}}{q_n^3}
    \leq\frac{e^{-\eta q_{n-1}}}{q_n^3},
\end{split}
\]
which proves \eqref{eq: weak EC consecutive gamma}.

Apply Lemma \ref{lem: folded subband lower bound} with $p=q_{n-1}$ and $q=q_n$. Since
\begin{align}
    \frac{\gamma_n}{e^{-c_1q_{n-1}}\ell_n^{-2}}
    &\leq
    \frac{e^{-\eta q_{n-1}}q_n^{-3}}
         {e^{-c_1q_{n-1}}(q_{n-1}/q_n)^2} 
     =\frac{e^{-(\eta-c_1)q_{n-1}}}{q_nq_{n-1}^2},\label{eq: key ratio}
\end{align}
we have $\gamma_n=o(e^{-c_1q_{n-1}}\ell_n^{-2})$ whenever $\eta>c_1$.
Thus the perturbative hypothesis in Lemma
\ref{lem: folded subband lower bound} holds on a sufficiently late
tail. After decreasing \(c_R\) once, every band of $\Sigma_n$ has
length at least
\begin{equation}\label{eq: mn definition}
    m_n\eqdef c_R e^{-c_1q_{n-1}}\ell_n^{-2}
    =c_R e^{-c_1q_{n-1}}\left(\frac{q_{n-1}}{q_n}\right)^2.
\end{equation}
Moreover, \eqref{eq: key ratio} and \(q_n\geq2q_{n-1}\) give
\begin{equation}\label{eq: gamma over mn summable}
    \sum_{n=1}^\infty \frac{\gamma_n}{m_n}<\infty.
\end{equation}
Indeed, the summands are bounded by a constant times
\[
    \frac{e^{-(\eta-c_1)q_{n-1}}}{q_nq_{n-1}^2},
\]
which is summable. We may also discard finitely many terms so that
the sequence $m_n$ is strictly decreasing to zero. Since \(q_{n+1}\geq2q_n\)
and \(q_n\geq2q_{n-1}\), this follows from
\begin{equation}
    \frac{m_{n+1}}{m_n}
    =e^{-c_1(q_n-q_{n-1})}\frac{q_n^4}{q_{n+1}^2q_{n-1}^2}<1
\end{equation}
for all large $n$; more explicitly, the ratio is at most
\[
    \frac{q_n^2}{4}e^{-c_1q_n/2},
\]
which tends to zero.

By \eqref{eq: gamma over mn summable}, after passing to a final tail, we may assume that, for every $n$,
\begin{equation}\label{eq: small tails}
    C\sum_{k=n+1}^{\infty}\frac{\gamma_k}{m_k}<\frac1{100},
    \qquad
    C\sum_{k=n+1}^{\infty}\gamma_k<\frac{m_{n+1}}{100},
\end{equation}
where \(C\geq1\) is the constant from Lemma
\ref{lem: local spectral loss}. The second inequality follows from
the first one and the monotonicity of $m_n$, since
\[
    \sum_{k=n+1}^{\infty}\gamma_k
    \leq m_{n+1}\sum_{k=n+1}^{\infty}\frac{\gamma_k}{m_k}.
\]
Set
\begin{equation}
    \delta_0\eqdef m_1.
\end{equation}

Fix $E\in\Sigma$ and $0<\delta\leq\delta_0$. Choose $n$ so that
\begin{equation}\label{eq: choose n by delta}
    m_{n+1}<\delta\leq m_n.
\end{equation}
Since $\sj_n\to\sj$ in operator norm,
\begin{equation}
    d_H(\Sigma,\Sigma_n)\leq\|\sj-\sj_n\|\leq \sum_{k=n+1}^\infty\gamma_k<\frac{\delta}{100},
\end{equation}
where we used \eqref{eq: small tails} and \eqref{eq: choose n by delta}. Hence there exists $E_n\in\Sigma_n$ with
\begin{equation}\label{eq: En close E}
    |E-E_n|<\frac{\delta}{100}.
\end{equation}
Let $B_n$ be a spectral band of $\Sigma_n$ containing $E_n$. By \eqref{eq: mn definition},
\begin{equation}
    |B_n|\geq m_n\geq\delta.
\end{equation}
Together with \eqref{eq: En close E}, this shows that the intersection
\(B_\delta(E)\cap B_n\) has length greater than
\(99\delta/100\). We can therefore choose a closed interval
\begin{equation}
    I\subset B_\delta(E)\cap B_n
\end{equation}
with
\begin{equation}\label{eq: I lower bound}
    |I|=\frac{9}{10}\delta.
\end{equation}

For $N>n$, the elementary inclusion
\begin{equation}
    \Sigma_n\setminus\Sigma_N\subseteq\bigcup_{k=n+1}^N(\Sigma_{k-1}\setminus\Sigma_k)
\end{equation}
gives, since $I\subset\Sigma_n$,
\begin{align}
    |I\cap\Sigma_N|
    &\geq |I|-\sum_{k=n+1}^N |I\cap(\Sigma_{k-1}\setminus\Sigma_k)|\nonumber\\
    &\geq |I|-C|I|\sum_{k=n+1}^N\frac{\gamma_k}{m_k}-C\sum_{k=n+1}^N\gamma_k,
\end{align}
where the last step uses Lemma \ref{lem: local spectral loss}. By
\eqref{eq: small tails}, \eqref{eq: choose n by delta}, and
\eqref{eq: I lower bound}, the first error term is at most
\(|I|/100\leq\delta/100\), while the second is at most
\(m_{n+1}/100<\delta/100\). Hence
\begin{equation}
    |I\cap\Sigma_N|
    \geq\frac{9}{10}\delta-\frac1{100}\delta-\frac1{100}\delta
    >\frac45\delta.
\end{equation}
Letting $N\to\infty$ and applying Proposition
\ref{prop: prop 2.1 in Filman} to the constant sequence \(I\) and the
Hausdorff-convergent sequence \(\Sigma_N\), we obtain
\begin{equation}
    |I\cap\Sigma|\geq\frac45\delta.
\end{equation}
Since $I\subset B_\delta(E)$, this gives
\begin{equation}
    |\Sigma\cap B_\delta(E)|\geq\frac45\delta.
\end{equation}
Thus $\Sigma$ is homogeneous with, for instance, homogeneity constant
$\tau=4/5$.  The proof shows that one may take
$\eta_0(R)=c_1(2R)$, with $c_1$ from Lemma
\ref{lem: uniform periodic band velocity}.
\end{proof}

\endgroup

\begingroup
\section{Reflectionlessness and pure absolute continuity}
\label{sec:pure-ac-proof}

For a whole-line Jacobi matrix \(J\) and \(k\in\Z\), let
\[
    g_k(z)=\langle\delta_k,(J-z)^{-1}\delta_k\rangle .
\]

\begin{definition}\label{def:reflectionless-corrected}
Let \(B\subseteq\R\) be a bounded Borel set.  A whole-line Jacobi
matrix \(J\) is reflectionless on \(B\), written \(J\in\mathcal R(B)\),
if
\[
    \operatorname{Re}g_k(t+i0)=0
\]
for Lebesgue almost every \(t\in B\) and every \(k\in\Z\).
\end{definition}

We use the following closure property.  For \(R_0>0\), write
\[
    \mathcal J_{R_0}=\{J:\|J\|\leq R_0\},
\]
and equip this set with the product topology on the Jacobi
coefficients.  The radius $R_0$ here bounds the operator norm; it is
independent of the coefficient radius $R$ in $\mathcal J(R)$.
In particular, a coefficient bound $2R$ gives an operator-norm bound $6R$,
not $2R$.

\begin{proposition}[Reflectionless closure]\label{prop:reflectionless-closure}
Suppose \(J_n,J\in\mathcal J_{R_0}\), \(J_n\to J\) in the product
topology, and \(B_n,B\subseteq\R\) are bounded Borel sets.  If
\[
    J_n\in\mathcal R(B_n)
    \quad\text{and}\quad
    |B\setminus B_n|\longrightarrow0,
\]
then \(J\in\mathcal R(B)\).
\end{proposition}

\begin{proof}
This is \cite[Proposition 3.1(a)]{PolRemApprox}.  Note that only the
one-sided measure condition displayed above is required.
\end{proof}

\begin{proposition}\label{prop:reflectionless-full-spectrum}
If \(\sj\in EC(\eta)\), then \(\sj\) is reflectionless on
\(\Sigma=\sigma(\sj)\).
\end{proposition}

\begin{proof}
Let
\[
    \sj_n=\sj_{q_n},\qquad
    \Sigma_n=\sigma(\sj_n),\qquad
    \varepsilon_n=\|\sj-\sj_n\|.
\]
Choose $R\geq1$ with $\sj\in\mathcal J(R)$.  As in the coefficient
estimate at the start of the homogeneity proof, after discarding finitely
many terms we have $\sj_n\in\mathcal J(2R)$.  Hence
\[
    \|\sj\|\leq3R,\qquad \|\sj_n\|\leq6R,
    \qquad \Sigma,\Sigma_n\subseteq[-6R,6R].
\]
Thus $R_0=6R$ is a common operator-norm bound, and all these spectra are
compact Borel sets.  This tail restriction preserves the exponential-class
condition and does not require $\eta$ to exceed the homogeneity threshold.

The spectral stability estimate gives
\begin{equation}\label{eq:Sigma-Sigma-n-Hausdorff}
    d_H(\Sigma,\Sigma_n)\leq\varepsilon_n.
\end{equation}
The spectrum \(\Sigma_n\) is a union of at most \(q_n\) compact
intervals.  Write
\[
    \mathscr N_\varepsilon(F)
    =\{x\in\R:\operatorname{dist}(x,F)\leq\varepsilon\}
\]
for the closed $\varepsilon$-neighborhood.  By
\eqref{eq:Sigma-Sigma-n-Hausdorff},
\[
    \Sigma\setminus\Sigma_n
    \subseteq
    \mathscr N_{\varepsilon_n}(\Sigma_n)\setminus\Sigma_n.
\]
The closed neighborhood is important here because the Hausdorff bound is
non-strict; it also covers $\varepsilon_n=0$.
Enlarging a union of at most \(q_n\) intervals by
\(\varepsilon_n\) adds total length at most
\(2q_n\varepsilon_n\).  Therefore
\begin{equation}\label{eq:one-sided-spectrum-measure}
    |\Sigma\setminus\Sigma_n|
    \leq2q_n\varepsilon_n = 2{q_n} \|\sj -\sj_n\| \les q_n^{-1}e^{-\eta q_n}\longrightarrow0.
\end{equation}
Every periodic Jacobi matrix is reflectionless on its spectrum \cite{PolRemApprox}, so
\[
    \sj_n\in\mathcal R(\Sigma_n).
\]
Moreover, the common bound $R_0=6R$ was established above, and norm
convergence implies convergence of all coefficients and hence convergence in
the product topology.  Proposition
\ref{prop:reflectionless-closure}, with \(B=\Sigma\) and
\(B_n=\Sigma_n\), now proves that
\(\sj\in\mathcal R(\Sigma)\).
\end{proof}

\begin{proposition}\label{prop:no-singular-corrected}
Suppose that a whole-line Jacobi matrix \(J\) is reflectionless on its
homogeneous spectrum \(\Sigma\).  Then \(J\) has no singular spectrum.
\end{proposition}

\begin{proof}
Let \(\mu_k\) be the diagonal spectral measure of
\((J,\delta_k)\), so that
\[
    g_k(z)=\int_{\R}\frac{d\mu_k(t)}{t-z}.
\]
By the reflectionlessness of \(J\), \(\mu_k\) is a reflectionless measure
on \(\Sigma\) in the notation of \cite{PolRem}.
By \cite[Theorem 1.1(a)]{PolRem}, the singular part
\((\mu_k)_s\) is carried by the zero-density set of \(\Sigma\), meaning
that $ (\mu_k)_s(\R\setminus Z_\Sigma)=0$, where
\begin{equation}\label{eq:density-zero-set-corrected}
    Z_\Sigma=
    \left\{x\in\R:
    \lim_{h\downarrow0}
    \frac{|\Sigma\cap(x-h,x+h)|}{2h}=0
    \right\}.
\end{equation}
If \(x\in\Sigma\), homogeneity gives, for all sufficiently small
\(h>0\),
\[
    \frac{|\Sigma\cap(x-h,x+h)|}{2h}\geq\frac{\tau}{2}>0.
\]
Thus \(Z_\Sigma\cap\Sigma=\varnothing\).  Since every \(\mu_k\) is
supported on \(\sigma(J)=\Sigma\) and $(\mu_k)_s\leq\mu_k$, the singular
part is also carried by $\Sigma$.  We conclude that
\((\mu_k)_s=0\) for every \(k\in\Z\).

Let \(P_s\) be the singular spectral projection of \(J\). Then
\[
    \|P_s\delta_k\|^2=(\mu_k)_s(\R)=0
\]
for every \(k\).  The linear span of the standard basis is dense in
\(\ell^2(\Z)\), so \(P_s=0\).
\end{proof}

\begin{proof}[Proof of Theorem \ref{thm:main}]
Choose \(\eta_0(R)\) as in Theorem \ref{Th: Homogeneity}.  If
\(\eta>\eta_0(R)\), that theorem shows that
\(\Sigma=\sigma(\sj)\) is homogeneous, while Proposition
\ref{prop:reflectionless-full-spectrum} shows that \(\sj\) is
reflectionless on \(\Sigma\).  Proposition
\ref{prop:no-singular-corrected} therefore gives \(P_s=0\), so
\(\sj\) has purely absolutely continuous spectral type.
\end{proof}

\section{Acknowledgments}

I would like to thank Millivoje Lukic for his comments on our previous work, which led to the writing of this manuscript. I would also like to thank Rui Han for helpful discussions, comments, and recommendations during the preparation of this manuscript.

\section{AI disclosure}
AI was used to check the manuscript for grammatical errors and to help improve the clarity of the exposition.
\endgroup

\end{document}